\documentclass{article}

\usepackage{mathptmx}
\usepackage[T1]{fontenc}
\usepackage{amsmath, amsthm, amsfonts}
\usepackage{amssymb}
\usepackage{graphicx}
\usepackage[font=footnotesize]{caption}
\usepackage{placeins} 

\newenvironment{nofigures}{
	\FloatBarrier       
	\suppressfloats[t]  
}{%
	\FloatBarrier       
}

\usepackage[left=3cm,right=4cm, top=4cm, bottom=3.5 cm]{geometry}
\numberwithin{equation}{section}
\newtheorem{thm}{Theorem}[section]

\newtheorem{lem}[thm]{Lemma}
\newtheorem{prop}[thm]{Proposition}
\theoremstyle{definition}
\newtheorem{defn}[thm]{Definition}
\theoremstyle{remark}
\newtheorem{rem}[thm]{Remark}

\newcommand\intervalN[2]{[#1\cdots #2)}
\newcommand\intervalC[2]{[#1\cdots #2]}

\def\Z{\mathbb{Z}}
\def\N{\mathbb{N}}
\def\ab{\hspace*{2pt}}
\def\positivenaturals{\mathbb{Z}^+}

\def\FormalLaurent{\mathbb{F}_2((x^{-1}))}
\def\F{\mathbb{F}}

\def\Coll{\mathsf{C}}

\title{On the average stopping time of the Collatz map in $\F_2[x]$}
\date{}
\author{Manuel Inselmann
}

\begin{document}
	\maketitle
	\begin{abstract}
		\noindent
	Define the map $T_1$ on $\mathbb{F}_2[x]$ by $T_1(f)=\frac{f}{x}$ if $f(0)=0$ and $T_1(f)=\frac{(x+1)f+1}{x}$ if $f(0)=1$. For a non-zero polynomial $f$ let $\tau_1(f)$ denote the least natural $k$ number for which $T_1^{k}(f)=1$. Define the average stopping time to be $\rho_1(n)=\frac{\sum_{f\in \mathbb{F}_2[x], \text{deg}(f)=n }\tau_1(f)}{2^n}$. We show that $\lim_{n\rightarrow\infty}\frac{\rho_1(n)}{n}=2$, confirming a conjecture of Alon, Behajaina, and Paran. Furthermore, we give a new proof that $\tau_1(f)\in O(\text{deg}(f)^{1.5})$ for all $f\in\mathbb{F}_2[x]\setminus\{0\}$.
	\end{abstract}
	\section{Introduction}
	The Collatz conjecture is a notoriously difficult problem, although it is very easy to state: Define the Collatz map $\Coll$ on the positive natural numbers by $\Coll(n)=\frac{n}{2}$ if $n$ is even and $\Coll(n)=3n+1$ if $n$ is odd. The conjecture states that $\Coll^k(n)$ will eventually reach $1$ as $k$ grows for every $n\in\positivenaturals$. Recently some progress has been made in \cite{Tao}, but the conjecture remains wide open. For an overview, see \cite{lagarias}.
	
	Instead of tackling the Collatz Conjecture it seems wise to deal with simpler problems first. In \cite{Hicks} a polynomial analogue of the Collatz map is considered. Define the \textit{polynomial Collatz map} $T_0:\F_2[x]\rightarrow \F_2[x]$ by $$T_0(f)=\begin{cases}
		\frac{f}{x}&\text{\ab if\ab} f(0)=0,\\
(x+1)f+1&\text{\ab if\ab} f(0)=1.
	\end{cases}$$
The \textit{stopping time} $\tau_0(f)$ of $f\in\F_2[x]\setminus\{0\}$ is the least $k\in\N$ such that $T^k_0(f)=1$. Hicks et al. showed that $\tau_0(f)$ is finite and bounded by $O(\deg(f)^2)$ (see, \cite{Hicks}). In \cite{alon} Alon et al. recently showed that  $\tau_0(f)$ is bounded by $O(\deg(f)^{1,5})$ thereby significantly improving the bound of Hicks et al. Furthermore, they introduced the \textit{average stopping time} $\rho_0(n)=\frac{\sum_{f\in \mathbb{F}_2[x], \deg(f)=n }\tau_0(f)}{2^{n}}$ and conjectured that $\rho_0(n)$ grows linearly in $n$. We confirm this conjecture by showing that  $\frac{\rho_0(n)}{n}$ converges to $3$ as $n\rightarrow\infty$. Instead of $T_0$ we base our analysis on an acceleration of it: Notice that if $f(0)=1$ then always $((x+1)f(x)+1)(0)=f(0)+1=0$, thus, $T_0(T_0(f))=\frac{(x+1)f+1}{x}$, whenever $f(0)=1$. So we consider the map $T_1:\F_2[x]\rightarrow \F_2[x]$ defined by $$T_1(f)=\begin{cases}
	\frac{f}{x}&\text{\ab if\ab} f(0)=0,\\
	\frac{(x+1)f+1}{x}&\text{\ab if\ab} f(0)=1.
\end{cases}$$\\
Again, define the \textit{stopping time} $\tau_1(f)$ to be the minimal $k\in\N$ such that $T_1^k(f)=1$. As $2\tau_1(f)-\deg(f)=\tau_0(f)$, it is enough to analyze the \textit{average stopping time} $\rho_1(n)$ defined by $\frac{\sum_{f\in \mathbb{F}_2[x], \deg(f)=n }\tau_1(f)}{2^{n}}$. We show that $\frac{\rho_1(n)}{n}$ converges to $2$ as $n\rightarrow\infty$. We do this by changing again to a different transformation, which is conjugate to $T_1$. The map $\sigma:\F_2[x]\rightarrow\F_2[x]: f(x)\mapsto f(x+1)$ is an automorphism on $\F_2[x]$, with $\sigma\circ\sigma=\text{id}_{\F_2[x]}$. The transformation of interest will be $T=\sigma\circ T_1\circ \sigma$. Again define  the \textit{stopping time} $\tau(f)$ to be the minimal $k\in\N$ such that $T^k(f)=1$ and the \textit{average stopping time} $\rho(n)$ to be $\frac{\sum_{f\in \mathbb{F}_2[x], \deg(f)=n }\tau(f)}{2^{n}}$. Since $T_1$ and $T$ are conjugate  and $\sigma(1)=1$, we have $\tau_1(f)=\tau(\sigma(f))$ and since $\deg(\sigma(f))=\deg(f)$, we have $\rho_1(n)=\rho(n)$ for all $n\in\N$.
Writing down $T$ explicitly, we obtain:
$$T(f)=\begin{cases}
	\frac{f}{x+1}&\text{\ab if\ab} f(1)=0,\\
	\frac{xf+1}{x+1}&\text{\ab if\ab} f(1)=1.
\end{cases}$$\\
Thus, the main result of this work is the following:
\begin{thm}\label{main}
		It holds that $\lim_{n\rightarrow\infty}\frac{\rho(n)}{n}=2$.
\end{thm} This theorem can be compared to the situation in the case of the Collatz map:
One can also define the map $\mathsf{C_2}$ as an acceleration of the Collatz map by setting $\mathsf{C_2}(n)=\frac{n}{2}$ if $n$ is even and $\mathsf{C_2}(n)=\frac{3n+1}{2}$ is $n$ is odd. For $n\in\positivenaturals$ define its \textbf{(total) stopping time} $\sigma(n)$ to be the minimal $k\in\N$ such that $C^k(n)=1$ if such a $k$ exists and set $\sigma(n)=\infty$ otherwise.  Crandall and Shanks conjectured that $\frac{1}{x}\sum_{n=1}^{\lfloor x\rfloor}\sigma(n)\sim 2(\log\frac{4}{3})^{-1}\log x$ (see \cite[page 13]{Lagariasgeneralizations}) -- or equivalently $\lim_{x\rightarrow\infty}\frac{1}{x\log x}\sum_{m=1}^{\lfloor x\rfloor}\tau(m)= 2(\log\frac{4}{3})^{-1}$. Thus, Theorem \ref{main} shows that the polynomial analogue of this Conjecture is true, where $\log n$ corresponds to the notion of the degree of a polynomial.

Furthermore, we will provide a new proof of the fact that $\tau(f)$ (and thus, $\tau_0(f)$) is bounded by $O(\deg(f)^{1.5})$ for all non-zero polynomials in $\F_2[x]$. 
   \section{On the average stopping time}  
 In the following, let $\FormalLaurent$ denote the field of formal power series of the form $\sum_{z\in\Z}a_zx^z$ with coefficients in $\F_2$ and $a_z\neq 0$ only for finitely many positive integers $z$ with addition defined by $$\sum_{z\in\Z}a_zx^z+ \sum_{z\in\Z}b_zx^z=\sum_{z\in\Z}(a_z+b_z)x^z$$
and multiplication defined by $$\left(\sum_{z\in\Z}a_zx^z\right)\left( \sum_{z\in\Z}b_zx^z\right)=\sum_{z\in\Z}\sum_{i+j=z}a_ib_jx^z.$$ Note that $\F_2[x]$ canonically is a subset of $\FormalLaurent$. Note also that $\FormalLaurent$ is isomorphic to the field $\F_2((x))$ of formal Laurent series over $\F_2$  by the isomorphism  $\sum_{z\in\Z}a_zx^z\mapsto \sum_{z\in\Z}a_zx^{-z}$ (we reversed the order to simplify notation). Let $\deg(\sum_{z\in\Z}a_zx^z)=\max\{z\mid a_z\neq 0\}$, in particular $\deg(0)=-\infty$. Note that $\deg(rs)=\deg(s)+\deg(r)$ for all $r,s\in \FormalLaurent$.  The primary observation is that $T$ is connected to a transformation on $\FormalLaurent$. We define $$P:\FormalLaurent\rightarrow \FormalLaurent; r\mapsto \frac{x}{x+1}r$$ (note that $\frac{x}{x+1}=\sum_{k=0}^\infty x^{-k}\in \FormalLaurent$).
 We also define the \textit{polynomial part map} $$[\cdot]:\FormalLaurent\rightarrow\F_2[x];\left[\sum_{z\in\Z}a_zx^z\right]=\sum_{z\in\N}a_kx^k.$$ Now, define the map $S: \FormalLaurent\rightarrow\FormalLaurent$ for $r\in\FormalLaurent$ by
$$S(r)=\begin{cases}
	\frac{r}{x+1}&\text{\ab if\ab} [r](1)=0,\\
	\frac{xr}{x+1}&\text{\ab if\ab} [r](1)=1.
\end{cases}$$\\
Or in short $S(r)=x^{[r](1)-1}P(r)$ (to abbreviate notation, we will assume that whenever $i\in\F_2$ appears in the exponent of a polynomial, the intended interpretation is as a natural number via the embedding $\F_2\rightarrow\N: 0\mapsto 0, 1\mapsto 1$).
In the following lemma we gather some useful properties.
\begin{lem}\label{gathering}
Suppose that $r,s\in\FormalLaurent$ and $r=\sum_{z\in\Z}a_zx^z$. 
\begin{enumerate}
	\item $P(r)=\sum_{z\in\Z}(\sum_{j=z}^\infty a_j)x^z$,
	\item $S(r)=\sum_{z\in\Z}(\sum_{j=z}^\infty a_j)x^{z+[r](1)-1}$,
	\item  $[r+s]=[r]+[s]$,
	\item $[P(r)]=[P([r])]$,
	\item $P^{-1}(r)=\sum_{k\in\Z}(a_{k+1}+a_{k})x^k$,
	\item If $[P(r)]=[P(s)]$, then $[r]=[s]$.
\end{enumerate}
\end{lem}
\begin{proof}
	To see $1.$ just note  $$P\left(\sum_{z\in\Z}a_zx^z\right)=\sum_{k=0}^\infty x^{-k}\sum_{z\in\Z}a_zx^z=\sum_{z\in\Z}\left(\sum_{j=z}^\infty a_j\right)x^z.$$ Then $2.$ follows from $1.$ as $S(r)=x^{[r](1)-1}P(r)$. Set $s=\sum_{z\in\Z}b_zx^z$. Then $3.$ is clear as $$[r+s]=\left[\sum_{z\in\Z}a_zx^z+\sum_{z\in\Z}b_zx^z\right]=\left[\sum_{z\in\Z}(a_z+b_z)x^z\right]=\sum_{z\in\N}(a_z+b_z)x^z=[r]+[s].$$ Now, $4.$ is a direct consequence of $1.$ as $$\left[P\left(\sum_{z\in\Z}a_zx^z\right)\right]=\left[\sum_{z\in\Z}\left(\sum_{j=z}^\infty a_j\right)x^z\right]=\sum_{z\in\N}\left(\sum_{j=z}^\infty a_j\right)x^z=[P([r])].$$ For $5.$ note that $(1+x^{-1})\left(\sum_{k\in\N}x^{-k}\right)=1$, thus, $P^{-1}(r)=(1+x^{-1})r$ and $$(1+x^{-1})\left(\sum_{z\in\Z}a_zx^z\right)=\sum_{k=0}^\infty (a_{z+1}+a_z)x^z.$$
	To see $6.$ let $r=\sum_{z\in\Z}a_zx^z$ and $s=\sum_{z\in\Z}b_zx^z$.  Then $[P(r)]=[P(s)]$ implies that $\sum_{j\geq n}a_j=\sum_{j\geq n}b_j$ for all $n\in\N$ by  $1.$ and from this one deduces that $a_k=b_k$ for all $k\geq 0$.
\end{proof}
A key observation is the following:
 \begin{lem}\label{compatible}
 	For every $r\in\FormalLaurent$ we have $T([r])=[S(r)]$ and inductively $T^n([r])=[S^n(r)]$ for all $n\in\N$.
 \end{lem}
\begin{proof}
	Let $r=\sum_{z\in\Z}a_zx^z\in\FormalLaurent$. There are two cases: If $[r](1)=0$, then $T([r])=\frac{[r]}{x+1}$ and $S(r)=\frac{r}{x+1}=\frac{[r]}{x+1}+\frac{r-[r]}{x+1}$. Now, $\deg(\frac{r-[r]}{x+1})< 0$, thus, $[S(r)]=\frac{[r]}{x+1}=T([r])$.\\
	 If $[r](1)=1$ then $$T([r])=\frac{x[r]+1}{x+1}=\frac{x[r]}{x+1}+\frac{1}{x+1}=P(r)+P([r]-r)+\frac{1}{x+1}.$$ As $\deg\left(P([r]-r)+\frac{1}{x+1}\right)< 0$, it follows that $\left[P(r)+P([r]-r)+\frac{1}{x+1}\right]=[P(r)]$, thus, $T([r])=[S(r)]$. For the second claim, if $T^k([r])=[S^k(r)]$, then $$T^{k+1}([r])=T(T^k([r]))=T([S^k(r)])=[S(S^k(r))]=[S^{k+1}(r)].$$
\end{proof}
Let $\FormalLaurent^+=\{r\in\FormalLaurent\mid \deg(r)\geq 0\}$. For $r\in\FormalLaurent^+$ define $\tau(r)$ to be the minimal natural number $n$ such that $[S^n(r)]=1$. Note that $\tau(r)$ is well-defined and coincides with the definition of $\tau$ when $[r]=r$ since $T^n(r)=[S^n(r)]$  by Lemma \ref{compatible}.
\begin{defn}
	Suppose that $r,s\in \FormalLaurent$. We write $r\preccurlyeq_xs$ if there exists $n\geq 0$ such that $x^nr=s$ (we also write $s\succcurlyeq_xr$ if $r\preccurlyeq_xs$). 
\end{defn}
Another key observation is the following:
\begin{lem}\label{shift}
	Suppose $r,s\in\FormalLaurent^+$. If $r\preccurlyeq_xs$, then also $S(r)\preccurlyeq_xS(s)$. In particular, $\tau(r)\leq\tau(s)$.
\end{lem}
\begin{proof}
		As $\preccurlyeq_x$ is transitive it is enough to note that $S(r)\preccurlyeq_x S(xr)$ for all $r\in\FormalLaurent^+$, which is immediate from the definition of $S$. Suppose that $[S^k(s)]=1$ holds. By what we have just shown $S^k(r)\preccurlyeq_xS^{k}(s)$. Note that if $\deg(r)\geq 0$ then also $\deg(S(r))\geq 0$, hence $[S^k(r)]=1$, thus, $\tau(r)\leq \tau(s)$.
\end{proof}
\begin{defn}\label{defmeasure}
	For $n\in\N$ define $P_n=\left\{f\in\F_2[x]\mid \deg(f)=n\right\}$ and let $\nu_n$ be the uniform probability measure on $P_n$, i.e., $\nu_n(\{f\})=\frac{1}{2^n}$. Set  $P_{\leq n}=\left\{f\in\F_2[x]\mid \deg(f)\leq n\right\}$. Note  that $(P_{\leq n}, +)$ is a finite group. If $\nu_{\leq n}$ is the Haar probability measure on $P_{\leq n}$, then $\nu_n$ is the restriction of $2\nu_{\leq n}$ to $P_n$. \\
	Now, define $Q_n=\left\{r\in\FormalLaurent\mid \deg(r)=n\right\}$ for $n\in\N$. Similarly as before we define a probability measure on $Q_n$. Set  $Q_{\leq n}=\left\{f\in\F_2[x]\mid \deg(f)\leq n\right\}$
	 Note  that $(Q_{\leq n}, +)$ is a compact group with respect to the norm $|f|=2^{\deg(f)}$. Let $\mu_{\leq n}$ be the Haar probability measure on $Q_{\leq n}$ and define $\mu_n$ to be the restriction of $2\mu_{\leq n}$ to $Q_n$. Since $Q_n\cup (x^n+Q_n)=Q_{\leq n}$ and  $Q_n\cap (x^n+Q_n)=\emptyset$, $\mu_n$ is a probability measure.
	
	 Define the sets $N^n_f$, for any $f\in \F_2[x]\setminus\{0\}$, by $N_f^n=\{r\in Q_n\mid [x^{\deg f-n}r]=f\}$. It is easy to determine $\mu_n(N^n_f)$. Just note that if $m=\deg(f)$ then $Q_n=\bigcup_{g\in P_{m}}N^n_g$ and the elements of the union are pairwise disjoint. Furthermore, $N_g=N_h+(g-h)x^{n-m}$ for all $g,h\in P_m$. Thus, $\mu_n(N^n_g)=\mu_n(N^n_h)$ for all $g,h\in P_m$ and hence $\mu_n(N^n_f)=2^{-m}$.
\end{defn}
Note that $[Q_n]=P_n$ and the push-forward measure of $\mu_n$ under the map $r\mapsto [r]$ is $\nu_n$.
By Lemma \ref{compatible} we know that $\tau(r)=\tau{([r])}$ thus, we have $\frac{\tau{(f)}}{2^{n}}=\int_{\{f\}}\tau d\nu_n=\int_{\{r\mid[r]=f\}}\tau d\mu_n=\int_{N^n_f}\tau d\mu_n$ for any $f\in P_n$. In particular, $\rho(n)=\frac{\sum_{f\in P_n }\tau(f)}{2^{n}}=\int_{Q_n}\tau d\mu_n$.
\begin{defn}
	Define $p(r)_k=[S^k(r)](1)$ for $r\in\FormalLaurent$.
	Let $\intervalN{a}{b}=[a,b)\cap\N$ and $\intervalC{a}{b}=[a,b]\cap\N$  for any $a,b\in\N$. The \textit{parity sequence of $r$ up to $n$} is the sequence  $\mathbf{p}^{r,n}\in\F_2^{\intervalN{0}{n}}$ defined by $\mathbf{p}^{r,n}_k=p(r)_k$ for all $0\leq k<n$.
\end{defn}
\begin{lem}\label{modulo}
	Suppose that $r\in \FormalLaurent$, $k\in\N$, and $f\in F_2[x]$ such that $f(1)=1$. Then for all $0\leq i< k$ we have $p(r)_i=p(r+(1+x)^kf)_i$, but $p(r)_k\neq p(r+(1+x)^kf)_k$. 
\end{lem}
\begin{proof}
	We proceed by induction on $k$. If $k=0$ then $$p(r+f)_0=[r+f](1)=[r](1)+[f](1)=[r](1)+1\neq [r](1)= p(r)_0.$$ Suppose we know the lemma for $0\leq k-1$. First, we have $$p(r)_0=[r](1)=[r+(1+x)^kf](1)=p(r+(1+x)^kf)_0.$$ Thus, $$S(r+(1+x)^kf)=\frac{x^{[r](1)}}{x+1}(r+(1+x)^kf)=\frac{x^{[r](1)}}{x+1}(r)+(1+x)^{k-1}x^{[r](1)}f=S(r)+(1+x)^{k-1}x^{[r](1)}f$$ and $\left(x^{[r](1)}f\right)(1)=1\cdot f(1)=1$. By induction hypothesis we have $p(S(r))_i=p(S(r+(1+x)^kf))_i$ for $0\leq i< k-1$ and $p(S(r))_{k-1}\neq p(S(r+(1+x)^kf))_{k-1}$. The induction step is complete by observing that $p(S(r))_i=p(r)_{i+1}$ by definition.
\end{proof}
\begin{lem}\label{prob}
	For every $\mathbf{p}\in\F_2^{\intervalN{0}{n}}$ the $\mu_n$-measure of the set $\{r\in Q_n\mid \mathbf{p}^{r,n}=\mathbf{p}\}$ equals $\frac{1}{2^n}$.
\end{lem}
\begin{proof}
First we show that  $f\in \{0\}\cup\bigcup_{k=0}^{n-1}P_{k}$ is uniquely determined by $\mathbf{p}^{f,n}$. Suppose that $f,g\in \{0\}\cup\bigcup_{k=0}^{n-1}P_{k}$ and $f\neq g$. We can write $g-f=(x+1)^kh$ with $0\leq k<n$ and $h\in\bigcup_{k=0}^{n-1}P_{k}$ with $h(1)=1$. Thus, we have $f=g+(x+1)^kh$ and by Lemma \ref{modulo} we conclude $p(g)_k\neq p(f)_k$, thus, $\mathbf{p}^{f,n}$ for $f\in \{0\}\cup\bigcup_{k=0}^{n-1}P_{n}$ are pairwise distinct. As there are exactly $2^n$ parity sequences up to $n$ and $\#(\{0\}\cup\bigcup_{k=0}^{n-1}P_k)=2^n$, for any given $\mathbf{p}\in\F_2^{\intervalN{0}{n}}$ we find exactly one $f\in\{0\}\cup\bigcup_{k=0}^{n-1}P_k$ with $\mathbf{p}^{f,n}=\mathbf{p}$. Then $(x+1)^n+f\in Q_n$ has the same parity sequence up to $n$ by Lemma \ref{modulo}.
 	By Lemma \ref{compatible} we have $[S^j(r)]=T^j([r])=[S^j([r])]$, thus, $\mathbf{p}^{r,n}=\mathbf{p}^{[r],n}$. Hence,
 	 $\mu_n\left(\left\{r\in Q_n\mid  \mathbf{p}^{r,n}=\mathbf{p}\right\}\right)\geq\mu_n(N^n_{(x+1)^n+f})=\frac{1}{2^n}$. But we have equality because this holds for every $\mathbf{p}\in\F_2^{\intervalN{0}{n}}$ and there are exactly $2^n$ of them - which completes the proof.  
\end{proof}
\begin{lem}\label{iterationS}
	We have $S^k(r)=\frac{x^{\sum_{i=0}^{k-1}p(r)_i}}{(x+1)^k}r$.
\end{lem}
\begin{proof}
	We proceed by induction. Obviously $S^0(r)=1\cdot r$. If $S^k(r)=\frac{x^{\sum_{i=0}^{k-1}p(r)_i}}{(x+1)^k}r$ then we obtain $S^{k+1}(r)=\frac{x^{[S^k(r)](1)}}{x+1}S^k(r)=\frac{x^{[S^k(r)](1)}}{x+1}\frac{x^{\sum_{i=0}^{k-1}p(r)_i}}{(x+1)^k}r=\frac{x^{\sum_{i=0}^{k}p(r)_i}}{(x+1)^{k+1}}r$.
\end{proof}
\begin{lem}\label{InvarianceMultiplication}
	Suppose that $k,l\in\Z$ and $n\in\N$ with $n+k-l\geq 0$. Then the push-forward of $\mu_n$ under the map $r\mapsto\frac{x^k}{(x+1)^l}r$ from $Q_n$ to $Q_{n+k-l}$ is $\mu_{n+k-l}$.
\end{lem}
\begin{proof}
We show the slightly stronger statement that if $s\in\FormalLaurent$, $n,m\geq 0$ and $n+\deg(s)=m$, then the push-forward of $\mu_n$ under the map $r\mapsto sr$ from $Q_n$ to $Q_{m}$ is $\mu_{m}$. With the notation of Definition \ref{defmeasure} we first note that $r\mapsto sr$ from $Q_{\leq n}$ to $Q_{\leq m}$ is a group isomorphism (the inverse is given by $r\mapsto \frac{r}{s}$). Thus, the push-forward of $2\mu_{\leq n}$ under this isomorphism is  $2\mu_{\leq m}$ by the uniqueness of the Haar-measure up to a constant. Furthermore, $s\cdot Q_n=Q_m$ and the claim follows.

Alternatively, one could go through the definitions in order to get an elementary proof: It suffices to show the claim for $k=1,l=0$ and $k=0,l=-1$. For $k=1,l=0$ note that $\{r\in Q_{n}\mid xr\in N^{n+1}_f\}=N^n_f$ and $\mu_{n+1}(N^{n+1}_f)=\mu_{n}(N^{n}_f)=\frac{1}{2^{\deg(f)}}$ for any $f\in\F_2[x]\setminus\{0\}$. Thus, the claim follows by the definition of the push-forward measure and the fact that a measure is determined by its values on the sets of the form $N^{n+1}_f$.
	For $k=0,l=-1$ we show that $\{r\in Q_{n}\mid (x+1)r\in N^{n+1}_f\}=N^n_{[\frac{xf}{x+1}]}$. We have $(x+1)r\in N^{n+1}_{f}$ iff $[(x+1)rx^{-n-1+\deg(f)}]=f$ iff $[rx^{-n+\deg(f)}]=[\frac{x}{x+1}f]$ iff $r\in N^n_{[\frac{xf}{x+1}]}$. To see the second ``iff'' note that $[(x+1)rx^{-n-1+\deg(f)}]=f=[f]$, iff $[P((x+1)rx^{-n-1+\deg(f)})]=[P[f]]$ iff $[rx^{-n+\deg(f)}]=[\frac{fx}{x+1}]$ by $6.$ of Lemma \ref{gathering}.
	Thus, the claim follows as in the previous case by noting that $\deg(f)=\deg\left(\left[\frac{xf}{x+1}\right]\right)$.
\end{proof}
\begin{lem}\label{Hoeffding}
	Suppose that $\epsilon>0$. Then $$\mu_n\left(\left\{r\in Q_n\mid \sum_{k=0}^{n-1}p(r)_k\geq n(\frac{1}{2}+\epsilon)\right\}\right)\leq e^{-2\epsilon^2n}$$ and  $$\mu_n\left(\left\{r\in Q_n\mid \sum_{k=0}^{n-1}p(r)_k\leq n(\frac{1}{2}-\epsilon)\right\}\right)\leq e^{-2\epsilon^2n}.$$
\end{lem}
\begin{proof}
 We will use Hoeffding's inequality, which states that if $\lambda_n$ is the uniform measure on $\left\{0,1\right\}^{\intervalN{0}{n}}$, then \begin{equation}\label{Hoeffding1}
		\lambda_{n}\left(\left\{x\in\left\{0,1\right\}^{\intervalN{0}{n}}\mid\sum_{k=0}^{n-1}x_k-\frac{n}{2}\geq\epsilon n\right\}\right)\leq e^{-2\epsilon^2n}.\end{equation}
	and
	\begin{equation}\label{Hoeffding2}
		\lambda_{n}\left(\left\{x\in\left\{0,1\right\}^{\intervalN{0}{n}}\mid\sum_{k=0}^{n-1}x_k-\frac{n}{2}\leq-\epsilon n\right\}\right)\leq e^{-2\epsilon^2n}.\end{equation}(see \cite{Hoeffding}).
		 Now, Lemma \ref{prob} states that the push-forward of $\mu_n$ under the map $r\mapsto \mathbf{p}^{r,n}$ is the uniform measure on $\left\{0,1\right\}^{\intervalN{0}{n}}$ and thus, we can apply Hoeffding's inequality to obtain the claim.
\end{proof}
\begin{prop}\label{upper}
	Suppose that $\epsilon>0$. Then there exists $n_0\in\N$ such that $\int_{Q_n}\tau d\mu_n\leq 2(1+\epsilon)n$ for all $n\geq n_0$. 
\end{prop}
\begin{proof}
	Choose $0<\delta<\frac{1}{2}$ and $\epsilon^\prime<\epsilon$ and fix $c>0$ such that $\tau(r)\leq c(\deg(r))^{1.5}$ for all $r\in\FormalLaurent^+$. Fix $n\in\N$. For $k\leq n$ define $$A_k=\left\{r\in Q_k\mid \sum_{i=0}^{k-1}p(r)_i\geq k\left(\frac{1}{2}+\delta\right)\right\},$$ then $\mu_k(A_k)\leq e^{-2\delta^2k}$ by Lemma \ref{Hoeffding}.
	 Recursively define $(k_0,j_0)=(n,0)$ and $$(k_{i+1},j_{i+1})=\left(\left\lfloor\left(\frac{1}{2}+\delta\right)k_i\right\rfloor,k_i+j_i\right).$$ Note that for $M>0$ we have $\left\lfloor\left(\frac{1}{2}+\delta\right)M\right\rfloor<M$ since $\delta<\frac{1}{2}$ -- thus, $k_{i+1}<k_i$ as long as $k_i>0$. Let $i_0$ be minimal such that $k_{i_0}<\sqrt{n}$. \\ 
	 Put $$B_i=\frac{(x+1)^{j_i}}{x^{j_{i+1}-n}}A_{k_i}\subseteq Q_n.$$ By Lemma \ref{InvarianceMultiplication} we have $$\mu_n(B_i)=\mu_{k_i}(A_{k_i})\leq e^{-2\delta^2k_i}.$$ Put $B^n=\bigcup_{i=0}^{i_0}B_i$ and put $G^n=Q_n\setminus B^n$. If $r\in G^n$ we show recursively that
	 $$S^{j_{i}}(r)\preccurlyeq_x\frac{x^{j_{i+1}-n}}{(x+1)^{j_{i}}}r$$ for $i\leq i_0$. 
	  This is certainly true for $i=0$. Suppose we have 
	   $$S^{j_{i}}(r)\preccurlyeq_x\frac{x^{j_{i+1}-n}}{(x+1)^{j_{i}}}r,$$ then, as $$\frac{x^{j_{i+1}-n}}{(x+1)^{j_{i}}}r\in \frac{x^{j_{i+1}-n}}{(x+1)^{j_{i}}}G^n\subseteq Q_{k_i}\setminus A_{k_i},$$ we obtain by hypothesis, Lemma \ref{shift}, Lemma \ref{iterationS}, and the definition of $A_{k_i}$
	  \begin{align*}S^{j_{i+1}}(r)=S^{j_i+k_i}(r)&=S^{k_i}(S^{j_i}(r))\preccurlyeq_x S^{k_i}\left(\frac{x^{j_{i+1}-n}}{(x+1)^{j_{i}}}r\right)\\=\frac{x^{\sum_{l=0}^{k_i-1}p\left(\frac{x^{j_{i+1}-n}}{(x+1)^{j_{i}}}r\right)_l}}{(x+1)^{j_i+k_i}}x^{j_{i+1}-n}r&\preccurlyeq_x \frac{x^{\left\lfloor\left(\frac{1}{2}+\delta\right)k_i\right\rfloor +j_{i+1}-n}}{(x+1)^{j_i+k_i}}r=\frac{x^{j_{i+2}-n}}{(x+1)^{j_{i+1}}}r.\end{align*}
	  In particular, $$S^{j_{i_0}}(r)\preccurlyeq_x\frac{x^{j_{i_0+1}-n}}{(x+1)^{j_{i_0}}}r.$$ 
	  Now, $$\deg\left(\frac{x^{j_{i_0+1}-n}}{(x+1)^{j_{i_0}}}r\right)=j_{i_0+1}-n+n-j_{i_0}=k_{i_0}<\sqrt{n}.$$ Thus, $\tau(r)\leq j_{i_0}+cn^{\frac{3}{4}}$ and $$j_{i_0}=\sum_{i=0}^{i_0-1}k_i\leq n\sum_{i=0}^{i_0-1} \left(\frac{1}{2}+\delta\right)^i\leq \frac{n}{\frac{1}{2}-\delta},$$ thus, $\tau(r)\leq \frac{2n}{1-2\delta}+cn^{\frac{3}{4}}\leq (2+\epsilon^\prime)n$ when $\delta>0$ is small enough  so that $\frac{2}{1-2\delta}<2+\epsilon^\prime$and then $n$ is large enough so that $\frac{2n}{1-2\delta}+cn^{\frac{3}{4}}<(2+\epsilon^\prime)n$. We are almost done. We just need to estimate what happens when $r\in B^n$. In that case we only know $\tau(r)\leq cn^{1.5}$. But this is enough as $B^n$ has small measure: Note that $$k_{i_0}\geq k_{i_0-1}\left(\frac{1}{2}+\delta\right)-1\geq \sqrt{n}\left(\frac{1}{2}+\delta\right)-1,$$ as $i_0$ is minimal such that $k_{i_0}<\sqrt{n}$. Thus, $$\mu_n(B^n)\leq \sum_{k=\lfloor\sqrt{n}(\frac{1}{2}+\delta)-1\rfloor}^{\infty}e^{-2\delta^2k}=e^{-2\delta^2\lfloor\sqrt{n}(\frac{1}{2}+\delta)-1\rfloor}\sum_{k=0}^{\infty}e^{-2\delta^2k}.$$ Hence, $$n^{1.5}\mu_n(B^n)\leq n^{1.5}e^{-2\delta^2\lfloor\sqrt{n}(\frac{1}{2}+\delta)-1\rfloor}\sum_{k=0}^{\infty}e^{-2\delta^2k},$$ thus, $c\cdot n^{1.5}\mu_n(B^n)\rightarrow 0$ as $n\rightarrow \infty$. In conclusion, $$\int_{Q_n}\tau d\mu_n=\int_{G^n}\tau d\mu_n+\int_{B^n}\tau d\mu_n\leq (2+\epsilon^\prime)n+c\cdot n^{1.5}\mu_n(B^n)\leq (2+\epsilon)n,$$ when $n$ is large enough since $\epsilon^\prime<\epsilon$.
\end{proof}
Similarly we get lower bounds:
\begin{prop}\label{lower}
	Suppose that $\epsilon>0$. Then there exists $n_0\in\N$ such that $\int_{Q_n}\tau d\mu_n\geq 2(1-\epsilon)n$ for all $n\geq n_0$. 
\end{prop}
\begin{proof}
	Choose $\delta,\epsilon^\prime,\epsilon^{\prime\prime}$ such that $0<\delta<\frac{1}{2}$, $2\epsilon^\prime<\epsilon^{\prime\prime}<\epsilon$ and choose $i_0$ such that $\left(\frac{1}{2}-\delta\right)^{i_0}<\epsilon^\prime$. Fix $n_0\in\N$ such that $n_0\left(\frac{1}{2}-\delta\right)^{i_0}>1$. Let $n\geq n_0$. For $k\leq n$ define $$A_k=\left\{r\in Q_k\mid \sum_{i=0}^{k-1}p(r)_i\leq k\left(\frac{1}{2}-\delta\right)\right\},$$ then $\mu_k(A_k)\leq e^{-2\delta^2k}$ by Lemma \ref{Hoeffding}.
	Recursively define $(k_0,j_0)=(n,0)$ and $$(k_{i+1},j_{i+1})=\left(\left\lceil\left(\frac{1}{2}-\delta\right)k_i\right\rceil,k_i+j_i\right).$$
		First, we show that $k_i\geq n\left(\frac{1}{2}-\delta\right)^i$ for $0\leq i\leq i_0$ -- just note that $k_0=n$ and inductively $$k_{i+1}=\left\lceil\left(\frac{1}{2}-\delta\right)k_i\right\rceil\geq \left(\frac{1}{2}-\delta\right)k_i\geq n\left(\frac{1}{2}-\delta\right)^{i+1}.$$
	  For $M\in\N$ and $M>1$ we have $\left\lceil\left(\frac{1}{2}-\delta\right)M\right\rceil<M$. As $$k_{i_0}\geq n\left(\frac{1}{2}-\delta\right)^{i_0}\geq n_0\left(\frac{1}{2}-\delta\right)^{i_0}>1$$ we obtain that $1<k_{i+1}<k_i$ for $0\leq i< i_0-1$.  \\ 
	Put $$B_i=\frac{(x+1)^{j_i}}{x^{j_{i+1}-n}}A_{k_i}\subseteq Q_n.$$ By Lemma \ref{InvarianceMultiplication} we have $$\mu_{n}(B_i)=\mu_{k_i}(A_{k_i})\leq e^{-2\delta^2k_i}.$$ Put $B^n=\bigcup_{i=0}^{i_0}B_i$ and put $G^n=Q_n\setminus B^n$. If $r\in G^n$ we show recursively that
	$$S^{j_{i}}(r)\succeq_x\frac{x^{j_{i+1}-n}}{(x+1)^{j_{i}}}r$$ for $i\leq i_0$. 
	This is certainly true for $i=0$. Suppose we have 
		$$S^{j_{i}}(r)\succeq_x\frac{x^{j_{i+1}-n}}{(x+1)^{j_{i}}}r.$$ As $$\frac{x^{j_{i+1}-n}}{(x+1)^{j_{i}}}r\in \frac{x^{j_{i+1}-n}}{(x+1)^{j_{i}}}G^n\subseteq Q_{k_i}\setminus A_{k_i},$$ we obtain by hypothesis, Lemma \ref{shift}, Lemma \ref{iterationS}, and the definition of $A_{k_i}$
	\begin{align*}
S^{j_{i+1}}(r)=S^{j_i+k_i}(r)&=S^{k_i}(S^{j_i}(r))\succeq_x S^{k_i}\left(\frac{x^{j_{i+1}-n}}{(x+1)^{j_{i}}}r\right)\\=\frac{x^{\sum_{l=0}^{k_i-1}p\left(\frac{x^{j_{i+1}-n}}{(x+1)^{j_{i}}}r\right)_l}}{(x+1)^{j_i+k_i}}x^{j_{i+1}-n}r&\succeq_x \frac{x^{\left\lceil\left(\frac{1}{2}-\delta\right)k_i\right\rceil +j_{i+1}-n}}{(x+1)^{j_i+k_i}}r=\frac{x^{j_{i+2}-n}}{(x+1)^{j_{i+1}}}r.	\end{align*}
	In particular, $$S^{j_{i_0}}(r)\succeq_x\frac{x^{j_{i_0+1}-n}}{(x+1)^{j_{i_0}}}r.$$ 

	Now, $$\deg\left(\frac{x^{j_{i_0+1}-n}}{(x+1)^{j_{i_0}}}r\right)=j_{i_0+1}-n+n-j_{i_0}=k_{i_0}\geq n\left(\frac{1}{2}-\delta\right)^{i_0}.$$ As $n\left(\frac{1}{2}-\delta\right)^{i_0}>0$, we obtain $\tau(r)\geq j_{i_0}$ and $$j_{i_0}=\sum_{i=0}^{i_0-1}k_i\geq n\sum_{i=0}^{i_0-1} \left(\frac{1}{2}-\delta\right)^i= n\frac{1-(\frac{1}{2}-\delta)^{i_0}}{\frac{1}{2}+\delta},$$ thus, $$\tau(r)\geq n\frac{1-(\frac{1}{2}-\delta)^{i_0}}{\frac{1}{2}+\delta}\geq n\frac{1-\epsilon^\prime}{\frac{1}{2}+\delta}\geq n(2-\epsilon^{\prime\prime})$$ when $\delta>0$ is small enough. We are almost done. We just need to estimate $\mu_n(B^n)$. As $$k_{i_0}\geq n\left(\frac{1}{2}-\delta\right)^{i_0}$$ we obtain $$ e^{-2\delta^2k_{i_0}}\leq e^{-2\delta^2\left(\frac{1}{2}-\delta\right)^{i_0}n}.$$ Thus, $$\mu(B^n)\leq \sum_{k=k_{i_0}}^{\infty}e^{-2\delta^2k}\leq e^{-2\delta^2\left(\frac{1}{2}-\delta\right)^{i_0}n}\sum_{k=0}^{\infty}e^{-2\delta^2k}=\frac{e^{-2\delta^2\left(\frac{1}{2}-\delta\right)^{i_0} n}}{1-e^{-2\delta^2}}.$$ In conclusion, $$\int_{Q_n}\tau d\mu_n\geq\int_{G^n}\tau d\mu_n\geq (2-\epsilon^{\prime\prime})n\left(1-\frac{e^{-2\delta^2\left(\frac{1}{2}-\delta\right)^{i_0} n}}{1-e^{-2\delta^2}}\right)\geq (2-\epsilon)n,$$ when $n$ is large enough since $\epsilon^{\prime\prime}<\epsilon$ and $\frac{e^{-2\delta^2(\frac{1}{2}-\delta)^{i_0} n}}{1-e^{-2\delta^2}}\rightarrow 0$ as $n\rightarrow\infty$.
\end{proof}
We now immediately get our main result:
\begin{thm}
It holds that $\lim_{n\rightarrow\infty}\frac{\rho(n)}{n}=2$ and $\lim_{n\rightarrow\infty}\frac{\rho_0(n)}{n}=3$.
\end{thm}
\begin{proof}
	The first part follows from Propositions \ref{upper} and \ref{lower} and the fact that $\rho(n)=\int_{Q_n}\tau d\mu_n$. The second part follows from the fact that for $f$ to reach $1$ under $T_1$ one needs exactly $\deg(f)$-many steps where the degree decreases, thus, $\tau_0(f)=\deg(f)+2(\tau_1(f)-\deg(f))=2\tau_1(f)-\deg(f)$, thus, $\rho_0(n)=2\rho_1(n)-n=2\rho(n)-n$, hence $\frac{\rho_0(n)}{n}=2\frac{\rho(n)}{n}-1\rightarrow 3$ as $n\rightarrow\infty$.
\end{proof}
This answers Conjecture~4.2 of ~\cite{alon}.
As an immediate corollary from the proofs of  Propositions \ref{upper} and \ref{lower} we note that most $f\in P_n$ reach $1$ in roughly $2n$-many steps:
\begin{thm}
	For every $\epsilon>0$ there exist $C_\epsilon,D_\epsilon>0$ such that for all $n\in\N$ $$\nu_n\left(\left\{f\in P_n\mid (2-\epsilon)n\leq\tau(f)\leq(2+\epsilon)n\right\}\right)\geq 1-C_\epsilon e^{-D_\epsilon\sqrt{n}}.$$
\end{thm}
And as direct consequence of this we note:
\begin{thm}
	The set of $r\in Q_0$ for which $\lim_{n\rightarrow\infty}\frac{\tau(x^nr)}{n}=2$ has $\mu_0$-measure $1$.
\end{thm}
\begin{proof}
Let $\epsilon>0$.	By the previous theorem and Lemma \ref{InvarianceMultiplication} we know that
there exist $C_\epsilon,D_\epsilon>0$ such that for all $n\in\N$ $$\mu_0\left(\left\{r\in Q_0\mid (2-\epsilon)n\leq\tau(x^nr)\leq(2+\epsilon)n\right\}\right)\geq 1-C_\epsilon e^{-D_\epsilon\sqrt{n}}.$$  Now, we have that $\sum_{n=0}^\infty C_\epsilon e^{-D_\epsilon\sqrt{n}}<\infty$. Put $A^\epsilon_N=\{r\in Q_0\mid \exists n\geq N: |\frac{\tau(x^nr)}{n}-2n|>\epsilon\}$. Then we conclude $\mu_0(A^\epsilon_N)\leq\sum_{n=N}^\infty C_\epsilon e^{-D_\epsilon\sqrt{n}}$. Put $$F_\epsilon=\left\{r\in Q_0\mid\exists N\in\N\forall n\geq N: |\frac{\tau(x^nr)}{n}-2|\leq\epsilon\right\}$$. Then $Q_0\setminus A^\epsilon_N\subseteq F_\epsilon$ for all $N\in\N$, thus, $\mu_0(F_\epsilon)=1$. We conclude that $$F=\bigcap_{n\in\N}F_{\frac{1}{n+1}}$$ also has full $\mu_0$-measure and if $r\in F$ then $\lim_{n\rightarrow\infty}\frac{\tau(x^nr)}{n}=2$.
\end{proof}
\section{An upper bound for the stopping time}
In this section we give an alternative proof of Theorem~1.1 in ~\cite{alon}. We begin with the following simple but crucial lemma.
\begin{lem}\label{periodicity}
	Suppose that $n\in\N$ and $r\in\FormalLaurent^+$ and $\deg(r)< 2^n$. Then $[r]=[P^{2^n}(r)]$.
\end{lem}
\begin{proof}
	Note that $(x+1)^{2^n}=1+x^{2^n}$ and $\frac{x^{2^n}}{1+x^{2^n}}=\sum_{k=0}^\infty x^{-k\cdot 2^n}$. Thus, $P^{2^n}(r)=\frac{x^{2^n}}{1+x^{2^n}}r=r+\sum_{k=1}^\infty x^{-k\cdot 2^n}r$ and since $\deg(\sum_{k=1}^\infty x^{-k\cdot 2^n}r)=\deg(r)-2^n< 0$, we get $\left[r+\sum_{k=1}^\infty x^{-k\cdot 2^n}r\right]=[r]$.
\end{proof}
\begin{lem}\label{one}
	Suppose that $r\in \FormalLaurent^+$. If $[r](0)=1$ or $[r](1)=1$, then $[S(r)](0)=1$.
\end{lem}
\begin{proof}
	Suppose that $\deg(r)=m$ and $[r]=\sum_{k=0}^ma_kx^{m-k}$. We have $S(r)=\frac{x^{p(r)_0}}{x+1}r$. 
		Suppose first that $[r](1)=1$, then $$[S(r)]=\left[\frac{x}{x+1}r\right]=\sum_{k=0}^{m}\left(\sum_{i=0}^{k}a_i\right)x^{m-k}$$ by part $2.$ of Lemma \ref{gathering}. Thus, $[S(r)](0)=\sum_{i=0}^ma_i=[r](1)=1$. It remains to consider the case in which $[r](0)=1$ and $[r](1)=0$. In this case $$[S(r)]=\left[\frac{1}{x+1}r\right]=\sum_{k=0}^{m-1}\left(\sum_{i=0}^{k}a_i\right)x^{m-1-k}$$ by part $2.$ of Lemma \ref{gathering}. Thus, $$[S(r)](0)=\sum_{i=0}^{m-1}a_i=\left(\sum_{i=0}^{m}a_i\right)+a_m=[r](1)+1=1$$ since $a_m=[r](0)=1$.
\end{proof}

We say that $\sum_{k=0}^{M}a_kx^{M-k}$ is an \textit{initial segment} of $\sum_{k=0}^{N}b_kx^{N-k}$ if $M\leq N$ and $a_k=b_k$ for all $0\leq k\leq M$. 
Suppose that $f\in\F_2[x]$ and let $n\in\N$ be such that $2^{n-1}\leq \deg(f)< 2^n$. For simplicity we assume $f(0)=1$. One main ingredient of the proof is that $T^{k+m\cdot 2^n}(f)$ is an initial segment of $[P^k(f)]$ for any $m\in\N$. In order to visualize the entire orbit of $f$ under $T$ consider the $[0\cdots 2^n-1]\times [0\cdots\deg(f)]$-matrix $A^f$, whose entry at $(k,i)$ is the coefficient of $x^{\deg(f)-i}$ of $[P^k(f)]$. Now, mark the index of the constant coefficient of each $T^k(f)$ (which is equal to $1$ by Lemma \ref{one} and the assumption that $f(0)=1$). If the degree of $T^k(f)$ does not change for $m$ many steps, say from $k_1$ to $k_1+m-1$, this implies that there is a vertical line consisting of $m$ marked entries in column $\deg(T^{k_1}(f))$. Now, there is a triangle of height $m-1$ and cardinality $\frac{m(m-1)}{2}$ to the left of this line of length $m$. In this triangle all entries are $0$ since $a_{i,j}=a_{i,j+1}+a_{i-1,j+1} \pmod 2$ by part $5.$ of Lemma \ref{gathering}. If we add the vertical line of length $m$ to the triangle we get a triangle with $\frac{m(m+1)}{2}$ many entries. For each column with index $1,\cdots,\deg(f)$ we get such a triangle, say of length $m_k$. Note that all these triangles are pairwise disjoint. Now, $\tau(f)$ is the sum of the lengths of the $\deg(f)$ many triangles (i.e., $\sum_{k=1}^{\deg(f)}m_k$). Thus, by the Cauchy-Schwarz inequality we obtain $$(\tau(f))^2=\left(\sum_{k=1}^{\deg(f)}m_k\right)^2\leq \deg(f)\sum_{k=1}^{\deg(f)} m_k^2\leq 2 \deg(f)\sum_{k=1}^{\deg(f)} \frac{1}{2}m_k(m_k+1)\leq 2 \deg(f)\cdot \deg(f)2^n,$$ thus, $\tau(f)\leq \deg(f)\sqrt{2\cdot 2^n}\leq 2\deg(f)^{1.5}$. As an illustration, see Figures \ref{figure 1} and \ref{figure 2}, where the above procedure is shown for a polynomial $P_L(x)$, which has a rather large stopping time and a polynomial $P_{R}$ chosen at random which has a stopping time $\approx 2\deg(P_{R})$.
\begin{figure}[ht]
	\centering
	\includegraphics[width=0.6\textwidth]{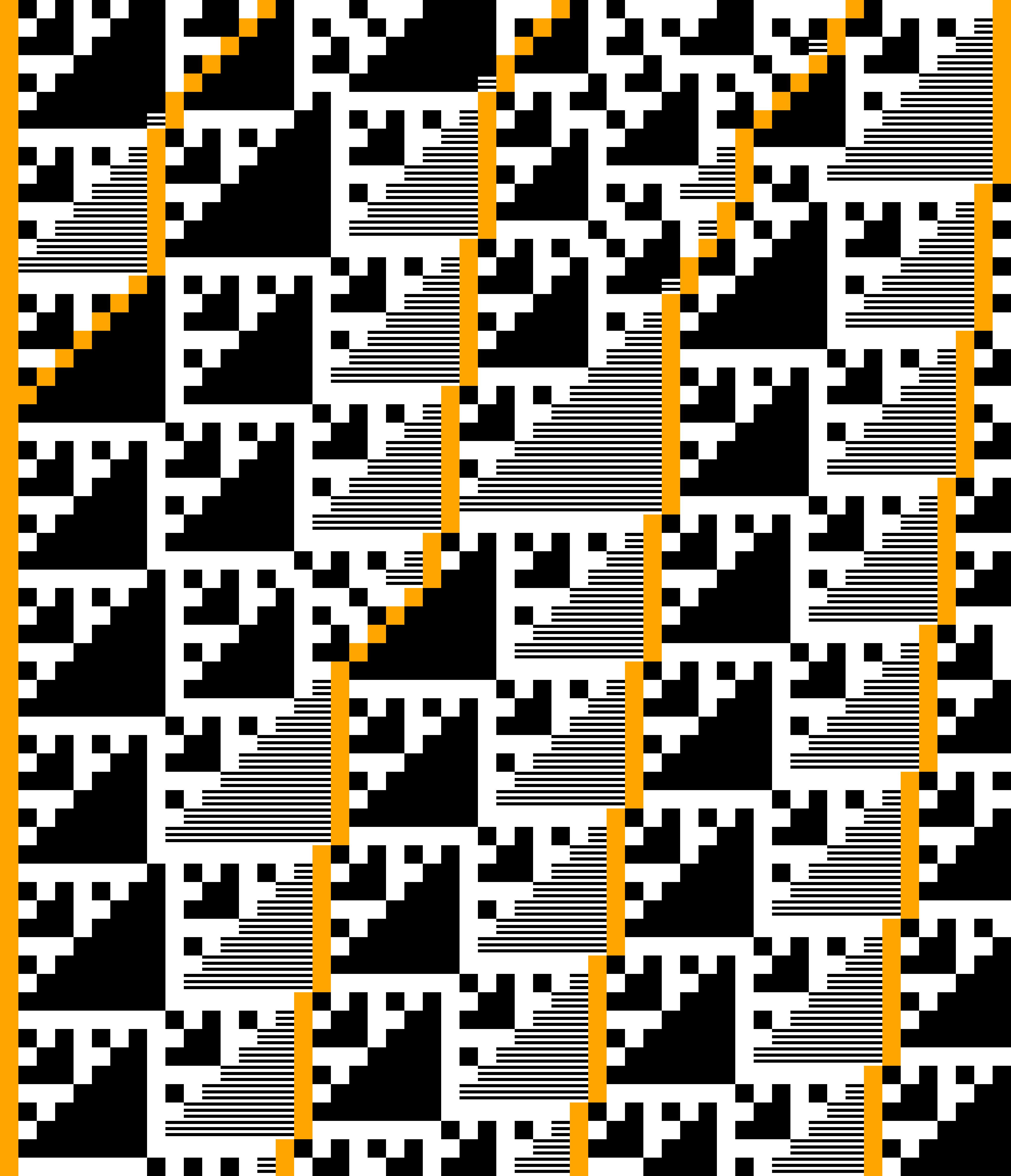}
	\captionsetup{width=0.6\textwidth}
	\caption{Illustration of the matrix $A^{P_L}$ for  $P_L(x)=x^{54} + x^{52} + x^{50} + x^{48} + x^{45} + x^{44} + x^{41} + x^{40} + x^{38} + x^{37} + x^{36} + x^{34} + x^{33} + x^{32} + x^{27} + x^{26} + x^{25} + x^{24} + x^{22} + x^{20} + x^{19} + x^{18} + x^{17} + x^{16} + x^{13} + x^{12} + x^{11} + x^{10} + x^{9} + x^{8} + x^{6} + x^{5} + x^{4} + x^{3} + x^{2} + x + 1$. The entries corresponding to the constant coefficients of $T^k(P_L)$ are colored orange, entries with zero belonging to a relevant triangle are striped black and white. All other cells with entry $0$ are colored black and all other cells with entry $1$ are colored white.}
	\label{figure 1}
\end{figure} In the following we present this argument formally. We begin with a few necessary definitions: 
\begin{defn}\label{visualmatrix} Suppose that $f\in\F_2[x]$ with $f(0)=1$.
	\begin{enumerate}
		\item For $k\in\N$ denote by $\overline{k}$ the unique $0\leq i<2^n$ such that $i\equiv k \pmod {2^n}$. 
	\item Define $A^f=(a_{i,j})_{0\leq i<2^n,0\leq j\leq \deg(f)}$ by $[P^i(f)]=\sum_{j=0}^{\deg (f)}a_{i,j}x^{\deg(f)-j}$ for $0\leq i<2^n$.

	  \item Let $I_f=\{(\overline{k},\deg(T^k(f)))\mid 0\leq k<\tau(f)\}$.
	 \item  Define $I_f^j=\{(i,j)\mid (i,j)\in I_f\}$ for $0< j\leq\deg(f)$.
	 \item For $0< j\leq \deg(f)$ define $i_j=\min\{k\in\N\mid\deg(T^k(f))=j\}$ and $h_j=\max\{k\in\N\mid\deg(T^k(f))=j\}$. 
	 \item Define $\Delta_j=\bigcup_{l=i_j}^{h_j}\{(\overline{l},j-d)\mid 0\leq d\leq l-i_j\}$ for $0< j\leq \deg(f)$. 
	 \item
	 For  $0< j\leq \deg(f)$ define $l(\Delta_j)=h_j-i_j+1$.
	\end{enumerate}
\end{defn} 
	Now, we gather the relevant results for the proof in a lemma.
	\begin{nofigures}
\begin{lem}\label{secondgathering}
	The following holds:
	\begin{enumerate}
		\item $T^k(f)=\sum_{j=0}^{\deg(T^k(f))}a_{\overline{k},j}x^{\deg(T^k(f))-j}$,
		\item $a_{\overline{k},\deg (T^{k}(f))}=1$ for every $k\geq 0$,
		\item for $(i,j)\in \intervalN{0}{2^n}\times\intervalC{0}{\deg(f)}$ and $m\leq \deg(f)-j$ we have $a_{i,j}=\sum_{l=0}^m\binom{m}{l}a_{\overline{i-l},j+m}$,
		\item for $(i,k)\in\Delta_j$ with $k>j$ we have $a_{i,k}=0$,
		\item for all $0\leq j<k\leq\deg(f)$ we have $\Delta_j\cap\Delta_k=\emptyset$,
		\item $\sum_{j=1}^{\deg(f)}l(\Delta_j)=\tau(f)$.
	\end{enumerate}
\end{lem}

\begin{proof}
	$1.$: By Lemma \ref{iterationS} we know that $S^k(f)=\frac{x^{\sum_{i=0}^{k-1}p(f)_i}}{(x+1)^k}f=x^{-k+\sum_{i=0}^{k-1}p(f)_i}P^{k}(f)$. Thus, by Lemma \ref{periodicity} $[S^k(f)]=[x^{-k+\sum_{i=0}^{k-1}p(f)_i}P^{\overline{k}}(f)]$. By Lemma \ref{compatible} we know that $T^k(f)=[S^k(f)]=\left[x^{-k+\sum_{i=0}^{k-1}p(f)_i}P^{\overline{k}}(f)\right]$. As $\deg(P^k(f))=\deg(f)$, we obtain $\deg(T^k(f))=-k+\sum_{i=0}^{k-1}p(f)_i+\deg(f)$. Thus, by definition, $$T^k(f)=\left[x^{-k+\sum_{i=0}^{k-1}p(f)_i}P^{\overline{k}}(f)\right]=\left[x^{\deg(T^k(f))-\deg(f)}\sum_{j=0}^{\deg (f)}a_{i,j}x^{\deg(f)-j}\right]=\sum_{j=0}^{\deg(T^k(f))}a_{\overline{k},j}x^{\deg(T^k(f))-j}.$$\\
	$2.$ is immediate by Lemma \ref{one} and $1.$ since $a_{\overline{k},\deg (T^{k}(f))}=T^{k}(f)(0)=1$ for every $k\geq 0$ as $f(0)=1$ by assumption.
	We prove $3.$ by induction on $m$: The case $m=0$ is trivial. By part $5.$ of Lemma \ref{gathering} we have $a_{\overline{i-1},j+1}=a_{i,j}+a_{i,j+1}$, thus, $a_{i,j}=a_{i,j+1}+a_{\overline{i-1},j+1}$. By induction hypothesis we have $a_{i,j+1}=\sum_{l=0}^m\binom{m}{l}a_{\overline{i-l},j+1+m}$ and 
	$a_{\overline{i-1},j+1}=\sum_{l=0}^m\binom{m}{l}a_{\overline{i-1-l},j+1+m}$, hence, \begin{align*}		
&a_{i,j}=a_{i,j+1}+a_{\overline{i-1},j+1}=\sum_{l=0}^m\binom{m}{l}a_{\overline{i-l},j+1+m}+\sum_{l=0}^m\binom{m}{l}a_{\overline{i-1-l},j+1+m}\\&=\sum_{l=0}^{m+1}\left(\binom{m}{l}+\binom{m}{l-1}\right)a_{\overline{i-l},j+1+m}=\sum_{l=0}^{m+1}\binom{m+1}{l}a_{\overline{i-l},j+(m+1)}.	\end{align*}
	To see $4.$ notice that by definition of $\Delta_j$ we have $(\overline{i-l},j)\in I_j$ for $0\leq l\leq j-k$ and therefore $a_{i-l,j}=1$ for $0\leq l\leq j-k$ by $2.$ Now, by $3.$ we have $a_{i,k}=\sum_{l=0}^{j-k}\binom{j-k}{l}a_{\overline{i-l},j}=2^{k-j}=0 \pmod 2$.\\ 
	To see $5.$, suppose that $\Delta_k\cap\Delta_j\neq\emptyset$ and $k<j$. Then $a_{\overline{l}, k}=1$ for $i_k\leq l\leq h_k$, thus, for at least one index $(b,k)$ in the intersection we must have $a_{b,k}=1$, a contradiction to $4.$ for $\Delta_j$.\\
	To see $6.$ note that by definition $\#I_f=\tau(f)$ and $\#I_f=\sum_{j=1}^{\deg(f)}\#I^j_f$. But $\#I_f^j=h_j-i_j+1$ and the claim follows.
\end{proof}
\end{nofigures}
Now, we are ready to give the promised proof:
\begin{thm}\label{Upperbound}
For all $f\in F_2[x]\setminus\{0\}$ it holds that	$\tau(f)\leq 2\deg(f)^{1.5}+1$.
\end{thm}
\begin{proof}
	We start by showing that it is sufficient to prove that $\tau(f)\leq 2\deg(f)^{1.5}$ for all $f\in\F_2[x]$ with $f(0)=1$: Given any $f\in\F_2[x]$ with $f(0)=0$ there exists a maximal $\tau(f)>k\geq 0$ such that $T^k(f)(0)=0$ (as $T^{\tau(f)}(f)(0)=1)$. By Lemma \ref{one} we have that $T^i(f)(1)=0$ for all $0\leq i<k$. Thus, $\deg(T^k(f))=\deg(f)-k$ and $\deg((T^{k+1}(f))\leq \deg(f)- k$. Hence, $$\tau(f)=k+1+\tau(T^{k+1}(f))\leq k +1 +2(\deg(f)-k)^{1.5}\leq 2\deg(f)^{1.5}+1.$$
	
Now, by part $6.$ of Lemma \ref{secondgathering} we have $\tau(f)=\sum_{j=0}^{\deg(f)}l(\Delta_j)$. By $5.$  of Lemma \ref{secondgathering} we conclude that $$\sum_{j=1}^{\deg(f)}\frac{l(\Delta_j)(l(\Delta_j)+1)}{2}=\sum_{j=1}^{\deg(f)}\#\Delta_j\leq 2^n\deg(f).$$
Thus, by the Cauchy-Schwarz inequality, we obtain (using $2^n\leq 2\deg(f)$) 
\begin{align*}
	(\tau(f))^2
	&=\left(\sum_{j=1}^{\deg(f)}l(\Delta_j)\right)^2
	\leq \deg(f)\cdot \sum_{j=1}^{\deg(f)}l(\Delta_j)^2
	\leq 2\deg(f)\sum_{j=0}^{\deg(f)}\frac{l(\Delta_j)(l(\Delta_j)+1)}{2}\\&\leq 2\deg(f)2^n\deg(f)
	\leq 4\cdot (\deg(f))^3.\end{align*}
	Thus, $\tau(f)\leq 2\deg(f)^{1.5}$.
\end{proof}
\begin{figure}[ht]
	\centering
	\includegraphics[width=0.6\textwidth]{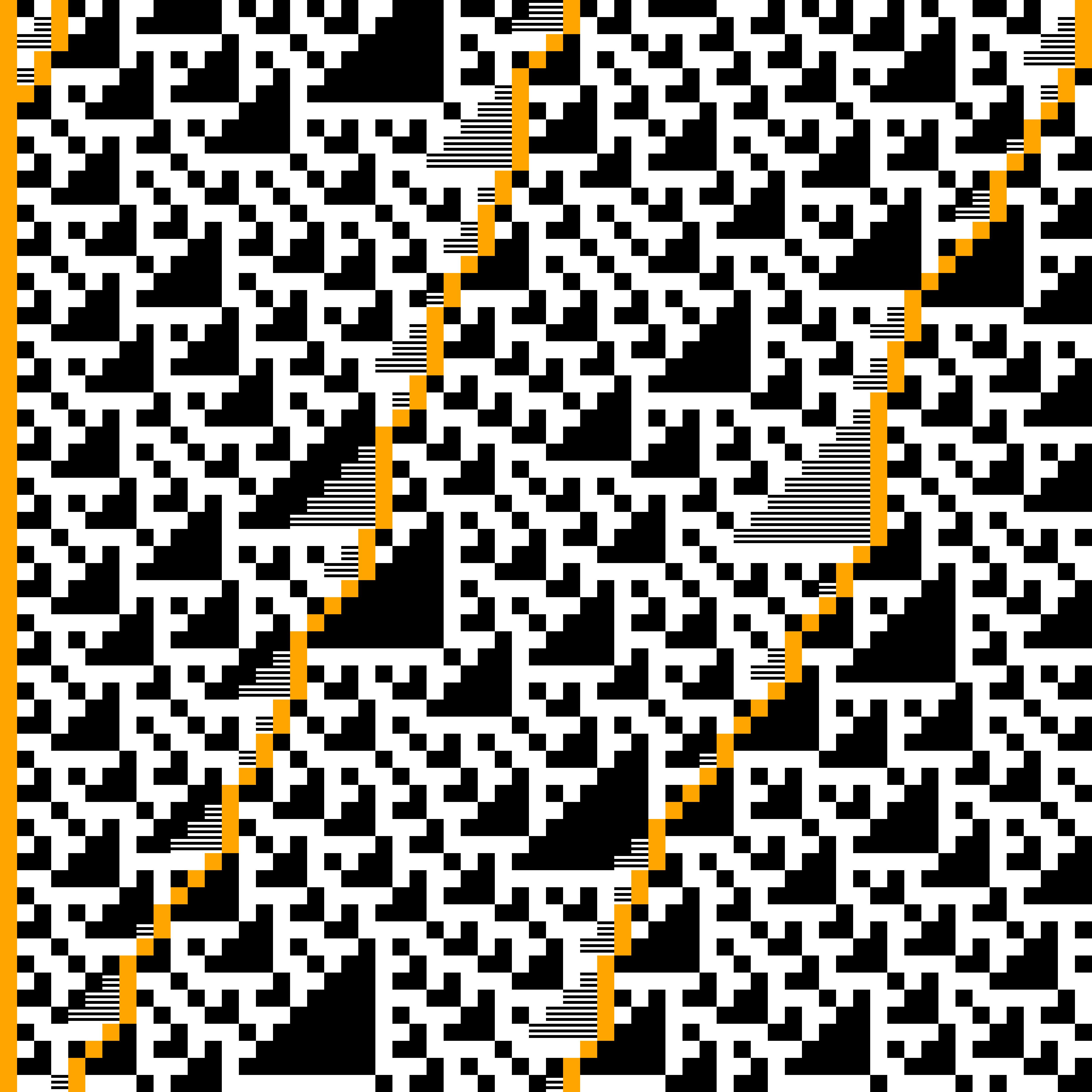}
	\captionsetup{width=0.6\textwidth} 
	\caption{Illustration of the matrix $A^{P_{R}}$ for  $P_{R}(x)=x^{63} + x^{61} + x^{60} + x^{58} + x^{56} + x^{55} + x^{50} + x^{48} + x^{46} + x^{44} + x^{42} + x^{41} + x^{37} + x^{34} + x^{30} + x^{28} + x^{26} + x^{21} + x^{20} + x^{19} + x^{17} + x^{15} + x^{13} + x^{12} + x^{10} + x^{8} + x^{7} + x^{4} + x^{2} + x + 1$. This is an example of a visualization as in Figure \ref{figure 1} with a polynomial chosen at random under the uniform distribution of all polynomials of degree $63$ with non-zero constant coefficient.}
	\label{figure 2}
\end{figure}
\begin{rem}
	It is possible to obtain slightly better bounds, e.g., $\tau(f)\leq \deg(f)^{1.5}+1$ by a more careful analysis. Also the proof suggests that one might want to look at $f$ such that many of the corresponding $\Delta_j$ have roughly length $\sqrt{2^n}$ in order to get a large stopping time.
\end{rem}
\begin{rem}
	Suppose that $f\in\F_2[x]$ and  $f(0)=1$. Look at $f^*=\frac{f+1}{x}$. Then $f(1)=j$ if and only if $f^*(1)=1+j$. Thus, if we look at the transformation 
		$T^*:\F_2[x]\rightarrow \F_2[x]$ defined by $$T^*(f)=\begin{cases}
			\frac{xf}{x+1}&\text{\ab if\ab} f(1)=0,\\
			\frac{f+1}{x+1}&\text{\ab if\ab} f(1)=1,
		\end{cases}$$
	then $T^*(f^*)=(T(f))^*$ for all
	 $f\in\F_2[x] $ with $f(0)=1$. 
\end{rem}
\begin{rem}
	Suppose that $f\in\F_2[x]$ and consider $A^f$ as defined in part 2. of Definition \ref{visualmatrix}. Then a straightforward induction on the degree of $f$ shows that $$f(x+1)=\sum_{k=0}^{\deg(f)}a_{\deg(f)-k+1,k}x^{\deg(f)-k}.$$
	This motivates the following transformation $U(A^f)$ defined by $U(A)_{i,j}=a_{\overline{\deg(f)-i-j},j}$, where again the row indices are understood to be mod $2^n$. We have $$f(x+1)=\sum_{k=0}^{\deg(f)}a_{\overline{\deg(f)+1-k},k}x^{\deg(f)-k}=\sum_{k=0}^{\deg(f)}U(A^f)_{\overline{-1},k}x^{\deg(f)-k}.$$
	Now, define $\hat{T}:\F_2[x]\rightarrow\F_2[x]$ by
	
	$$\hat{T}(f)=\begin{cases}
		\frac{f+f(0)}{x}&\text{\ab if\ab} [P^{-1}(f)](0)=0,\\
		\frac{(x+1)f+f(0)}{x}&\text{\ab if\ab} [P^{-1}(f)](0)=1,
	\end{cases}$$
or shorter $\hat{T}(f)=[P(T_1([P^{-1}(f)]))]$.
Note that the visualization of the orbit of $f$ in $A^f$ under $T$ corresponds under $U$ to the corresponding visualization of the orbit of $\sum_{k=0}^{\deg(f)}U(A^f)_{0,k}x^{\deg(f)-k}$ under $\hat{T}$.

  $\hat{T}$ might be used to find polynomials with large stopping time. E.g., take $P_n=\sum_{k=0}^{2^n-1}x^{4^n-k\cdot (2^n+1)-1}$. Computations for small $n$ suggest that $\hat{\tau}(P_n)\approx n\cdot 4^n$, where $\hat{\tau}(f)$ is the least $n\in \N$ such that $(\hat{T})^n(f)=1$.
\end{rem}
\begin{rem}\label{generalization}
	It is unclear whether Theorem \ref{Upperbound} gives a tight upper bound for the stopping time. In a slightly more general setting it does: For positive natural numbers $n,m$ we define $M_C^{n,m}$ to be the set of all $\intervalN{0}{n}\times \intervalC{0}{m}$-matrices $A$ with entries in $\F_2$ with the following property: For all $0\leq i<n$ and $0\leq j<m$ we have $a_{i,j}=a_{i,j+1}+a_{i-1,j+1}$ (with the convention $a_{-1,j}=a_{n-1,j}$). Consider the sets $A_1=\{(i,j)\in\intervalN{0}{n}\times\intervalC{0}{m}\mid a_{i,j}=1\}$ and $A^+_1=\{(i,j)\in A_1\mid j>0\}\cup\{(i,0)\in A_1\mid a_{i-1,0}=1\}$. One can define a map $T_A: A^+_1\rightarrow A_1$ by $$T_A((i,j))=\begin{cases}
		(i-1,j)&\textit{\ab if\ab} a_{i-1,j}=1\\
	(i,j-1)&\textit{\ab if\ab} a_{i-1,j}=0.\\
	\end{cases}$$
	Given $(i_0,j_0)\in A^+_1$ the stopping time $\tau_A((i_0,j_0))$ is the least $k\in\N$ with $$p\left(T^k_A((i_0,j_0))\right)=\min_{n\in\N}p\left(T^n_A((i_0,j_0))\right),$$ where $p((i,j))=j$. For example, for $f\in \F_2[x]$ we have $\tau(f)=\tau_{U(A^{f})}((0,\deg(f)))$. Given $n,m$ define $\tau_{n,m}=\max\{\tau_A(n-1,m)\mid A\in M_C^{n,m} \wedge (n-1,m)\in A_1\}$. Without going into further detail here, there exists $C>0$ such that $\tau_{n,n-1}\geq C\cdot n^{1.5}$ for infinitely many $n\in\N$ (see Figure \ref{figure 3}).  
	\begin{figure}[ht]
		\centering
		\includegraphics[width=0.6\textwidth]{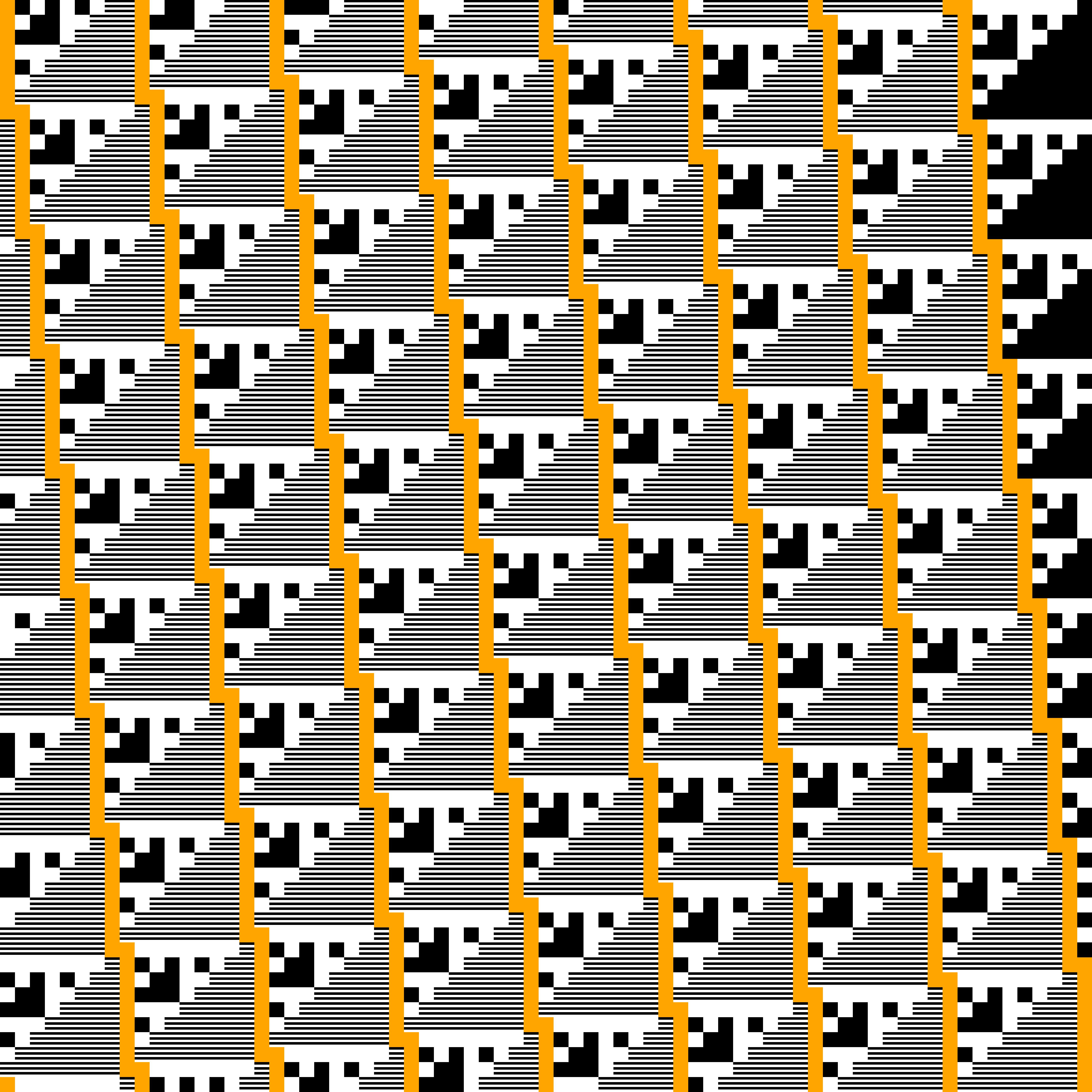}
		\captionsetup{width=0.6\textwidth} 
		\caption{Illustration of the orbit of $T_A((72,72))$ for a matrix $A$ for which $\tau_A((72,72))$ is close to $73^{1.5}$.}
		\label{figure 3}
	\end{figure} 
\end{rem}
\begin{rem}
	Let $p\in\N$ be prime. Consider the map $$T_p:\F_p[x]\rightarrow \F_p[x]; f\mapsto\begin{cases}
		\frac{xf+f(-1)}{x+1}&\text{\ab if\ab} f(-1)=f(0)\neq 0\\
		\frac{f-f(-1)}{x+1}&\text{\ab else} .
	\end{cases}$$ It is straightforward to see that for any $f\in\F_p[x]\setminus \{0\}$ there exist $n\in\N$ such that $\deg(T_p^n(f))=0$. Let $\tau_p(f)$ be the minimal $n$ with this property.  An analogous argument as in this section can be applied to show that there exist $C_p>0$ such that $\tau_p(f)\leq C_p (\deg(f))^{1.5}$ for all $f\in\F_p[x]\setminus \{0\}$.
\end{rem}
\begin{nofigures}
\subsection*{Acknowledgments}
The author thanks Angelot Behajaina for many valuable comments on an earlier version of this paper and Claudius Röhl for explicit computations of stopping times.

	\bibliography{Voll}
	\bibliographystyle{plain}
	\textit{email address:} Manuel.Inselmann@gmx.de
\end{nofigures}
\end{document}